\documentclass[12pt, a4paper]{article}

\newcommand{\quick}[1]{#1} 

\usepackage[a4paper]{geometry}
\usepackage{pgfplots}
\pgfplotsset{width=7cm,compat=1.14}

\usepackage{url}

\renewcommand{\labelenumi}{(\alph{enumi})}
\renewcommand\theenumi\labelenumi

\usepackage[english]{babel}

\usepackage[utf8]{inputenc}
\usepackage{xspace}
\usepackage{amsmath,amsthm,amssymb,mathtools}
\usepackage{lmodern}

\usepackage[algo2e,ruled,vlined,linesnumbered]{algorithm2e}

\usepackage{xcolor}
\usepackage{tikz}
\usepackage{graphicx}

\allowdisplaybreaks[4]
\newtheorem{theorem}{Theorem}
\newtheorem{lemma}[theorem]{Lemma}
\newtheorem{corollary}[theorem]{Corollary}

\newcommand{\N}{\ensuremath{\mathbb{N}}} 
\newcommand{\Z}{\ensuremath{\mathbb{Z}}}

\DeclareMathOperator{\Bin}{Bin}

\newcommand{\eps}{\varepsilon}

\begin{document}
\sloppy

\title{An Elementary Analysis of the Probability \\That a Binomial Random Variable \\Exceeds Its Expectation}

\author{Benjamin Doerr\\ Laboratoire d'Informatique (LIX)\\ \'Ecole Polytechnique \\ Palaiseau \\ France}

\maketitle

\begin{abstract}
  We give an elementary proof of the fact that a binomial random variable $X$ with parameters $n$ and $0.29/n \le p < 1$ with probability at least $1/4$ strictly exceeds its expectation. We also show that for $1/n \le p < 1 - 1/n$, $X$ exceeds its expectation by more than one with probability at least $0.0370$. Both probabilities approach $1/2$ when $np$ and $n(1-p)$ tend to infinity.
\end{abstract}

\section{Introduction}

Let $X$ be a random variable following a binomial distribution with parameters $n$ and $p$, that is, we have $\Pr[X=i] = \binom ni p^i (1-p)^{n-i}$ for all $i \in [0..n]$. Then, apart from maybe extreme cases, it seems very natural that with reasonable probability $X$ is at least its expectation $E[X]$ or even exceeds it. Surprisingly, and despite the fact that such statements are very important in the machine learning literature, only very recently rigorous proofs of such statements appeared. We refer to Greenberg and Mohri~\cite{GreenbergM14} for a detailed discussion on the previous lack of such results.

Prior to the work of Greenberg and Mohri, apart from general bounds like those in Slud~\cite{Slud77}, apparently only a result of Rigollet and Tong~\cite{RigolletT11} was known. This result is stated as $\Pr[X \ge E[X]] \ge \min\{p,\frac 14\}$ for all $p \le \frac 12$ in the paper (Lemma~6.4), but the proof shows the stronger statement 
\begin{equation}
\Pr[X \ge E[X]] \ge 
\begin{cases}
\;\tfrac 14 &\mbox{if $p \in [\tfrac 1n, \tfrac 12]$}\\
\;p &\mbox{if $p < \tfrac 1n$} \, .
\end{cases}
\label{eqrigollettrue}
\end{equation}
The main work in the proof is showing another interesting result, namely that for all $k \in [2..\frac n2]$ one has 
\begin{equation}
\Pr[\Bin(n,\tfrac kn) \ge k+1] \ge \Pr[\Bin(n,\tfrac {k-1}n) \ge k].\label{eqrigolletarg}
\end{equation}
The proof of this result uses a connection between binomial distributions and order statistics of uniform distributions (to be found in Section 7.2 of the second volume of Feller~\cite{Feller71}) and then proceeds by showing the inequality \[k \int_0^{\frac{k-1}{n}} t^{k-1} (1-t)^{n-k} \, dt \le (n-k) \int_0^{\frac{k}{n}} t^{k} (1-t)^{n-k-1} \, dt.\] 

It is not clear how to extend~\eqref{eqrigollettrue} to $p > \frac 12$. Note that neither~\eqref{eqrigolletarg} nor this equation with the inequality reversed are true for all $k \in [\frac n2..n-1]$. Hence the following relatively recent result of Greenberg and Mohri appears to be the first one treating the problem in full generality.

\begin{lemma}[Greenberg and Mohri~\cite{GreenbergM14}]\label{lprobfeigebin1}
  Let $n \in \N$ and $\frac 1n < p \le 1$. Let $X \sim \Bin(n,p)$. Then 
  \begin{align*}
  &\Pr[X \ge E[X]] > \tfrac 14.
  \end{align*}  
\end{lemma}

This result has found applications not only in machine learning, but also in randomized algorithms, see, e.g.,~\cite{KarppaKK16,BecchettiCNPST17,MitzenmacherM17}. While the result is very simple, the proof is not and uses the Camp-Paulson normal approximation to the binomial cumulative distribution function. 

Via a different, again non-elementary proof technique, using among others the hazard rate order and the likelihood ratio order of integer-valued distributions, the following result was shown by Pelekis and Ramon~\cite{PelekisR16}.

\begin{lemma}[Pelekis and Ramon~\cite{PelekisR16}]\label{lprobfeigebin2}
  Let $n \in \N$ and $\frac 1n \le p \le 1-\frac 1n$. Let $X \sim \Bin(n,p)$. Then 
  \begin{align*}
  &\Pr[X \ge E[X]] \ge \frac{1}{2\sqrt 2} \, \frac{\sqrt{np(1-p)}}{\sqrt{np(1-p)+1}+1}\,.
  \end{align*}  
\end{lemma}

Lemma~\ref{lprobfeigebin2} improves the bound of Lemma~\ref{lprobfeigebin1} when $np(1-p) > 8$, which in particular requires $n > 32$ and $E[X] = np > 8$. It however never gives a bound better than $\frac{1}{2\sqrt 2} \approx 0.3536$. 

In this work, we show that also truly elementary arguments give interesting results for this problem. We prove in Lemma~\ref{lmain} that for $\frac 1n \le p < 1$ and $k := \lfloor np \rfloor$, we have 
  \begin{align}
  \Pr[X > E[X]] 
  &> \frac 12 - \sqrt{\frac{n}{2\pi k (n - k)}} \,.\label{eq:ourresult}
  \end{align}  
This bound is not perfectly comparable to the previous, but Figure~\ref{figbounds} indicates that it is often superior. It has the particular advantage that it tends to~$\frac 12$ when $np$ and $n(1-p)$ tend to infinity. Our bound does not immediately imply the $\frac 14$ bound of Greenberg and Mohri~\cite{GreenbergM14}, however elementary analyses of a few ``small cases'' suffice to obtain in Theorem~\ref{tmain} that $\Pr[X > E[X]] \ge \frac 14$ for all $0.2877 \frac1n \approx \ln(\frac43) \frac1n \le p < 1$. The strict version $\Pr[X > E[X]] > \frac 14$ of the claim is also valid except when $n=2$ and $p = \frac 12$.

We also show that for $\frac 1n \le p < 1 - \frac 1n$, $X$ exceeds its expectation by more than one with probability at least $0.0370$, again with better bounds available when $np$ and $n(1-p)$ are larger, see Theorem~\ref{tplusone}. Such a statement was recently needed in the analysis of an evolutionary algorithm (in the proof of Lemma~3 of the extended version of~\cite{DoerrGWY17}).

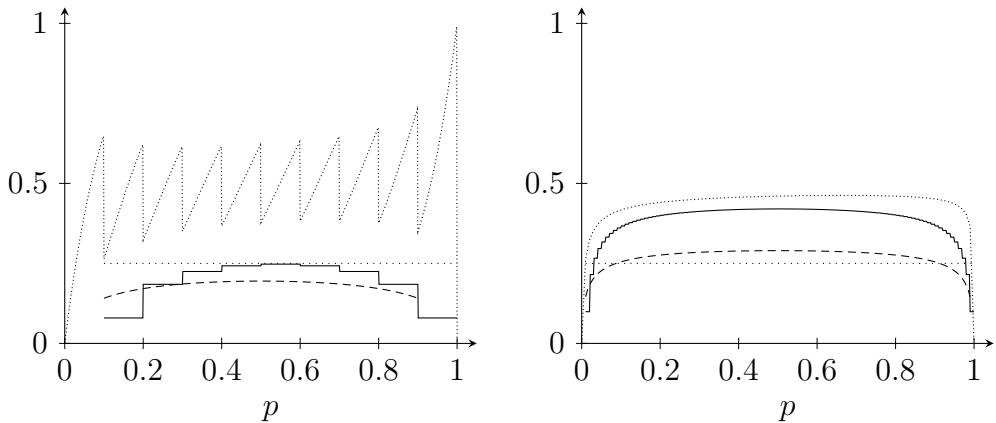
\begin{figure}
\begin{tikzpicture}
\def\nn{10}
\begin{axis}[
    x tick style={color=black},
    y tick style={color=black},
    axis lines = left,
    xlabel = $p$,
    legend style={
      cells={anchor=east},
      legend pos=outer north east,},
    xmin=0,
    xmax=1.05,
    ymin=0.00,
    ymax=1.05,
    domain=0:1
]

\addplot [solid,
    samples=1000, 
    color=black,
    domain=1/\nn:1,
    ]
    {(1/2)-sqrt(\nn/(2*pi*floor(\nn*x)*(\nn-floor(\nn*x))))};
%
\addplot [dotted,
    samples=10, 
    color=black,
    domain=1/\nn:1,
    ]
    {1/4};
%
\addplot [densely dashed,
    samples=200, 
    color=black,
    domain=1/\nn:1-1/\nn,
    ]
    {sqrt(\nn*x*(1-x))/(2*sqrt(2)*(sqrt(\nn*x*(1-x)+1)+1))};

\addplot [densely dotted,
    color=black,
    domain=0:1,
    ]
table{
x y
0	0
0.001	0.00995512
0.002	0.019820957
0.003	0.029598223
0.004	0.039287626
0.005	0.04888987
0.006	0.05840565
0.007	0.06783566
0.008	0.077180588
0.009	0.086441117
0.01	0.095617925
0.011	0.104711686
0.012	0.113723068
0.013	0.122652735
0.014	0.131501347
0.015	0.140269558
0.016	0.148958018
0.017	0.157567373
0.018	0.166098264
0.019	0.174551327
0.02	0.182927193
0.021	0.19122649
0.022	0.199449841
0.023	0.207597865
0.024	0.215671174
0.025	0.223670379
0.026	0.231596085
0.027	0.239448893
0.028	0.2472294
0.029	0.254938197
0.03	0.262575873
0.031	0.270143012
0.032	0.277640193
0.033	0.285067993
0.034	0.292426981
0.035	0.299717726
0.036	0.30694079
0.037	0.314096733
0.038	0.32118611
0.039	0.328209471
0.04	0.335167364
0.041	0.342060331
0.042	0.348888913
0.043	0.355653643
0.044	0.362355053
0.045	0.36899367
0.046	0.375570019
0.047	0.382084618
0.048	0.388537984
0.049	0.394930629
0.05	0.401263061
0.051	0.407535785
0.052	0.413749301
0.053	0.419904108
0.054	0.426000698
0.055	0.432039562
0.056	0.438021187
0.057	0.443946054
0.058	0.449814643
0.059	0.455627429
0.06	0.461384886
0.061	0.467087481
0.062	0.47273568
0.063	0.478329944
0.064	0.483870732
0.065	0.489358498
0.066	0.494793695
0.067	0.500176769
0.068	0.505508167
0.069	0.510788328
0.07	0.516017693
0.071	0.521196695
0.072	0.526325765
0.073	0.531405333
0.074	0.536435824
0.075	0.541417659
0.076	0.546351257
0.077	0.551237033
0.078	0.556075402
0.079	0.56086677
0.08	0.565611546
0.081	0.570310131
0.082	0.574962927
0.083	0.57957033
0.084	0.584132735
0.085	0.588650531
0.086	0.593124109
0.087	0.597553852
0.088	0.601940142
0.089	0.60628336
0.09	0.610583882
0.091	0.61484208
0.092	0.619058327
0.093	0.623232988
0.094	0.62736643
0.095	0.631459015
0.096	0.635511102
0.097	0.639523047
0.098	0.643495205
0.099	0.647427926
0.1	0.263901071
0.101	0.267777358
0.102	0.271657533
0.103	0.275541184
0.104	0.279427904
0.105	0.283317296
0.106	0.287208965
0.107	0.291102525
0.108	0.294997593
0.109	0.298893793
0.11	0.302790757
0.111	0.306688119
0.112	0.310585521
0.113	0.314482609
0.114	0.318379038
0.115	0.322274464
0.116	0.326168551
0.117	0.330060968
0.118	0.333951389
0.119	0.337839493
0.12	0.341724966
0.121	0.345607496
0.122	0.34948678
0.123	0.353362516
0.124	0.35723441
0.125	0.361102172
0.126	0.364965517
0.127	0.368824163
0.128	0.372677837
0.129	0.376526266
0.13	0.380369185
0.131	0.384206333
0.132	0.388037452
0.133	0.39186229
0.134	0.3956806
0.135	0.399492138
0.136	0.403296665
0.137	0.407093946
0.138	0.410883752
0.139	0.414665856
0.14	0.418440037
0.141	0.422206076
0.142	0.425963761
0.143	0.429712882
0.144	0.433453233
0.145	0.437184613
0.146	0.440906825
0.147	0.444619675
0.148	0.448322974
0.149	0.452016535
0.15	0.455700176
0.151	0.45937372
0.152	0.463036991
0.153	0.466689819
0.154	0.470332035
0.155	0.473963477
0.156	0.477583984
0.157	0.481193399
0.158	0.484791568
0.159	0.488378342
0.16	0.491953574
0.161	0.49551712
0.162	0.499068841
0.163	0.502608601
0.164	0.506136264
0.165	0.509651702
0.166	0.513154786
0.167	0.516645393
0.168	0.520123402
0.169	0.523588694
0.17	0.527041154
0.171	0.530480671
0.172	0.533907135
0.173	0.53732044
0.174	0.540720482
0.175	0.544107161
0.176	0.547480379
0.177	0.550840041
0.178	0.554186054
0.179	0.557518329
0.18	0.560836778
0.181	0.564141317
0.182	0.567431865
0.183	0.570708341
0.184	0.57397067
0.185	0.577218776
0.186	0.580452588
0.187	0.583672036
0.188	0.586877053
0.189	0.590067575
0.19	0.593243538
0.191	0.596404884
0.192	0.599551554
0.193	0.602683492
0.194	0.605800645
0.195	0.608902962
0.196	0.611990394
0.197	0.615062894
0.198	0.618120417
0.199	0.621162919
0.2	0.322200474
0.201	0.32522223
0.202	0.328247583
0.203	0.331276353
0.204	0.334308364
0.205	0.337343438
0.206	0.340381401
0.207	0.343422076
0.208	0.346465291
0.209	0.349510871
0.21	0.352558644
0.211	0.355608439
0.212	0.358660085
0.213	0.361713412
0.214	0.364768251
0.215	0.367824435
0.216	0.370881795
0.217	0.373940168
0.218	0.376999386
0.219	0.380059286
0.22	0.383119706
0.221	0.386180482
0.222	0.389241454
0.223	0.392302461
0.224	0.395363345
0.225	0.398423947
0.226	0.401484111
0.227	0.404543679
0.228	0.407602499
0.229	0.410660414
0.23	0.413717274
0.231	0.416772925
0.232	0.419827219
0.233	0.422880004
0.234	0.425931132
0.235	0.428980457
0.236	0.432027832
0.237	0.435073112
0.238	0.438116153
0.239	0.441156812
0.24	0.444194949
0.241	0.447230421
0.242	0.45026309
0.243	0.453292818
0.244	0.456319468
0.245	0.459342903
0.246	0.46236299
0.247	0.465379594
0.248	0.468392584
0.249	0.471401828
0.25	0.474407196
0.251	0.47740856
0.252	0.480405792
0.253	0.483398765
0.254	0.486387356
0.255	0.489371438
0.256	0.492350891
0.257	0.495325592
0.258	0.498295421
0.259	0.501260259
0.26	0.504219988
0.261	0.507174491
0.262	0.510123654
0.263	0.513067361
0.264	0.5160055
0.265	0.518937958
0.266	0.521864627
0.267	0.524785395
0.268	0.527700155
0.269	0.5306088
0.27	0.533511224
0.271	0.536407324
0.272	0.539296995
0.273	0.542180136
0.274	0.545056646
0.275	0.547926425
0.276	0.550789375
0.277	0.5536454
0.278	0.556494402
0.279	0.559336288
0.28	0.562170964
0.281	0.564998338
0.282	0.567818319
0.283	0.570630817
0.284	0.573435743
0.285	0.57623301
0.286	0.579022532
0.287	0.581804224
0.288	0.584578002
0.289	0.587343784
0.29	0.590101489
0.291	0.592851035
0.292	0.595592345
0.293	0.59832534
0.294	0.601049943
0.295	0.603766081
0.296	0.606473677
0.297	0.60917266
0.298	0.611862958
0.299	0.614544498
0.3	0.350389282
0.301	0.353059447
0.302	0.355733309
0.303	0.35841075
0.304	0.361091654
0.305	0.363775906
0.306	0.366463388
0.307	0.369153986
0.308	0.371847582
0.309	0.374544062
0.31	0.377243308
0.311	0.379945206
0.312	0.38264964
0.313	0.385356494
0.314	0.388065652
0.315	0.390777
0.316	0.393490422
0.317	0.396205803
0.318	0.398923028
0.319	0.401641983
0.32	0.404362553
0.321	0.407084624
0.322	0.409808081
0.323	0.412532811
0.324	0.4152587
0.325	0.417985635
0.326	0.420713502
0.327	0.423442189
0.328	0.426171583
0.329	0.428901571
0.33	0.431632042
0.331	0.434362883
0.332	0.437093983
0.333	0.439825231
0.334	0.442556515
0.335	0.445287727
0.336	0.448018754
0.337	0.450749487
0.338	0.453479818
0.339	0.456209636
0.34	0.458938832
0.341	0.461667299
0.342	0.464394929
0.343	0.467121613
0.344	0.469847245
0.345	0.472571717
0.346	0.475294924
0.347	0.47801676
0.348	0.480737118
0.349	0.483455894
0.35	0.486172984
0.351	0.488888282
0.352	0.491601686
0.353	0.494313093
0.354	0.497022398
0.355	0.499729502
0.356	0.5024343
0.357	0.505136693
0.358	0.50783658
0.359	0.51053386
0.36	0.513228434
0.361	0.515920203
0.362	0.518609068
0.363	0.52129493
0.364	0.523977693
0.365	0.52665726
0.366	0.529333534
0.367	0.532006419
0.368	0.53467582
0.369	0.537341643
0.37	0.540003793
0.371	0.542662177
0.372	0.545316701
0.373	0.547967275
0.374	0.550613806
0.375	0.553256202
0.376	0.555894375
0.377	0.558528233
0.378	0.561157687
0.379	0.56378265
0.38	0.566403033
0.381	0.569018749
0.382	0.57162971
0.383	0.574235833
0.384	0.57683703
0.385	0.579433218
0.386	0.582024312
0.387	0.584610229
0.388	0.587190886
0.389	0.589766203
0.39	0.592336096
0.391	0.594900486
0.392	0.597459293
0.393	0.600012438
0.394	0.602559841
0.395	0.605101426
0.396	0.607637116
0.397	0.610166833
0.398	0.612690502
0.399	0.615208049
0.4	0.366896742
0.401	0.369407044
0.402	0.371921435
0.403	0.374439825
0.404	0.376962123
0.405	0.379488236
0.406	0.382018073
0.407	0.384551543
0.408	0.387088553
0.409	0.38962901
0.41	0.392172822
0.411	0.394719896
0.412	0.397270139
0.413	0.399823458
0.414	0.402379759
0.415	0.404938949
0.416	0.407500934
0.417	0.41006562
0.418	0.412632913
0.419	0.415202719
0.42	0.417774943
0.421	0.420349492
0.422	0.42292627
0.423	0.425505183
0.424	0.428086136
0.425	0.430669034
0.426	0.433253783
0.427	0.435840286
0.428	0.438428451
0.429	0.44101818
0.43	0.443609379
0.431	0.446201953
0.432	0.448795806
0.433	0.451390844
0.434	0.45398697
0.435	0.456584091
0.436	0.459182109
0.437	0.461780931
0.438	0.46438046
0.439	0.466980602
0.44	0.469581262
0.441	0.472182344
0.442	0.474783752
0.443	0.477385393
0.444	0.479987171
0.445	0.482588991
0.446	0.485190759
0.447	0.487792379
0.448	0.490393757
0.449	0.492994798
0.45	0.495595408
0.451	0.498195493
0.452	0.500794959
0.453	0.503393711
0.454	0.505991656
0.455	0.508588699
0.456	0.511184748
0.457	0.513779708
0.458	0.516373488
0.459	0.518965992
0.46	0.52155713
0.461	0.524146807
0.462	0.526734932
0.463	0.529321413
0.464	0.531906156
0.465	0.534489072
0.466	0.537070067
0.467	0.539649051
0.468	0.542225932
0.469	0.54480062
0.47	0.547373025
0.471	0.549943055
0.472	0.552510621
0.473	0.555075634
0.474	0.557638003
0.475	0.56019764
0.476	0.562754456
0.477	0.565308362
0.478	0.567859271
0.479	0.570407094
0.48	0.572951744
0.481	0.575493134
0.482	0.578031176
0.483	0.580565785
0.484	0.583096874
0.485	0.585624358
0.486	0.588148152
0.487	0.590668169
0.488	0.593184326
0.489	0.595696539
0.49	0.598204723
0.491	0.600708796
0.492	0.603208674
0.493	0.605704275
0.494	0.608195516
0.495	0.610682317
0.496	0.613164595
0.497	0.61564227
0.498	0.618115262
0.499	0.62058349
0.5	0.376953125
0.501	0.37941651
0.502	0.381884738
0.503	0.38435773
0.504	0.386835405
0.505	0.389317683
0.506	0.391804484
0.507	0.394295725
0.508	0.396791326
0.509	0.399291204
0.51	0.401795277
0.511	0.404303461
0.512	0.406815674
0.513	0.409331831
0.514	0.411851848
0.515	0.414375642
0.516	0.416903126
0.517	0.419434215
0.518	0.421968824
0.519	0.424506866
0.52	0.427048256
0.521	0.429592906
0.522	0.432140729
0.523	0.434691638
0.524	0.437245544
0.525	0.43980236
0.526	0.442361997
0.527	0.444924366
0.528	0.447489379
0.529	0.450056945
0.53	0.452626975
0.531	0.45519938
0.532	0.457774068
0.533	0.460350949
0.534	0.462929933
0.535	0.465510928
0.536	0.468093844
0.537	0.470678587
0.538	0.473265068
0.539	0.475853193
0.54	0.47844287
0.541	0.481034008
0.542	0.483626512
0.543	0.486220292
0.544	0.488815252
0.545	0.491411301
0.546	0.494008344
0.547	0.496606289
0.548	0.499205041
0.549	0.501804507
0.55	0.504404592
0.551	0.507005202
0.552	0.509606243
0.553	0.512207621
0.554	0.514809241
0.555	0.517411009
0.556	0.520012829
0.557	0.522614607
0.558	0.525216248
0.559	0.527817656
0.56	0.530418738
0.561	0.533019398
0.562	0.53561954
0.563	0.538219069
0.564	0.540817891
0.565	0.543415909
0.566	0.54601303
0.567	0.548609156
0.568	0.551204194
0.569	0.553798047
0.57	0.556390621
0.571	0.55898182
0.572	0.561571549
0.573	0.564159714
0.574	0.566746217
0.575	0.569330966
0.576	0.571913864
0.577	0.574494817
0.578	0.57707373
0.579	0.579650508
0.58	0.582225057
0.581	0.584797281
0.582	0.587367087
0.583	0.58993438
0.584	0.592499066
0.585	0.595061051
0.586	0.597620241
0.587	0.600176542
0.588	0.602729861
0.589	0.605280104
0.59	0.607827178
0.591	0.61037099
0.592	0.612911447
0.593	0.615448457
0.594	0.617981927
0.595	0.620511764
0.596	0.623037877
0.597	0.625560175
0.598	0.628078565
0.599	0.630592956
0.6	0.382280602
0.601	0.384791951
0.602	0.387309498
0.603	0.389833167
0.604	0.392362884
0.605	0.394898574
0.606	0.397440159
0.607	0.399987562
0.608	0.402540707
0.609	0.405099514
0.61	0.407663904
0.611	0.410233797
0.612	0.412809114
0.613	0.415389771
0.614	0.417975688
0.615	0.420566782
0.616	0.42316297
0.617	0.425764167
0.618	0.42837029
0.619	0.430981251
0.62	0.433596967
0.621	0.43621735
0.622	0.438842313
0.623	0.441471767
0.624	0.444105625
0.625	0.446743798
0.626	0.449386194
0.627	0.452032725
0.628	0.454683299
0.629	0.457337823
0.63	0.459996207
0.631	0.462658357
0.632	0.46532418
0.633	0.467993581
0.634	0.470666466
0.635	0.47334274
0.636	0.476022307
0.637	0.47870507
0.638	0.481390932
0.639	0.484079797
0.64	0.486771566
0.641	0.48946614
0.642	0.49216342
0.643	0.494863307
0.644	0.4975657
0.645	0.500270498
0.646	0.502977602
0.647	0.505686907
0.648	0.508398314
0.649	0.511111718
0.65	0.513827016
0.651	0.516544106
0.652	0.519262882
0.653	0.52198324
0.654	0.524705076
0.655	0.527428283
0.656	0.530152755
0.657	0.532878387
0.658	0.535605071
0.659	0.538332701
0.66	0.541061168
0.661	0.543790364
0.662	0.546520182
0.663	0.549250513
0.664	0.551981246
0.665	0.554712273
0.666	0.557443485
0.667	0.560174769
0.668	0.562906017
0.669	0.565637117
0.67	0.568367958
0.671	0.571098429
0.672	0.573828417
0.673	0.576557811
0.674	0.579286498
0.675	0.582014365
0.676	0.5847413
0.677	0.587467189
0.678	0.590191919
0.679	0.592915376
0.68	0.595637447
0.681	0.598358017
0.682	0.601076972
0.683	0.603794197
0.684	0.606509578
0.685	0.609223
0.686	0.611934348
0.687	0.614643506
0.688	0.61735036
0.689	0.620054794
0.69	0.622756692
0.691	0.625455938
0.692	0.628152418
0.693	0.630846014
0.694	0.633536612
0.695	0.636224094
0.696	0.638908346
0.697	0.64158925
0.698	0.644266691
0.699	0.646940553
0.7	0.382782786
0.701	0.385455502
0.702	0.388137042
0.703	0.39082734
0.704	0.393526323
0.705	0.396233919
0.706	0.398950057
0.707	0.40167466
0.708	0.404407655
0.709	0.407148965
0.71	0.409898511
0.711	0.412656216
0.712	0.415421998
0.713	0.418195776
0.714	0.420977468
0.715	0.42376699
0.716	0.426564257
0.717	0.429369183
0.718	0.432181681
0.719	0.435001662
0.72	0.437829036
0.721	0.440663712
0.722	0.443505598
0.723	0.4463546
0.724	0.449210625
0.725	0.452073575
0.726	0.454943354
0.727	0.457819864
0.728	0.460703005
0.729	0.463592676
0.73	0.466488776
0.731	0.4693912
0.732	0.472299845
0.733	0.475214605
0.734	0.478135373
0.735	0.481062042
0.736	0.4839945
0.737	0.486932639
0.738	0.489876346
0.739	0.492825509
0.74	0.495780012
0.741	0.498739741
0.742	0.501704579
0.743	0.504674408
0.744	0.507649109
0.745	0.510628562
0.746	0.513612644
0.747	0.516601235
0.748	0.519594208
0.749	0.52259144
0.75	0.525592804
0.751	0.528598172
0.752	0.531607416
0.753	0.534620406
0.754	0.53763701
0.755	0.540657097
0.756	0.543680532
0.757	0.546707182
0.758	0.54973691
0.759	0.552769579
0.76	0.555805051
0.761	0.558843188
0.762	0.561883847
0.763	0.564926888
0.764	0.567972168
0.765	0.571019543
0.766	0.574068868
0.767	0.577119996
0.768	0.580172781
0.769	0.583227075
0.77	0.586282726
0.771	0.589339586
0.772	0.592397501
0.773	0.595456321
0.774	0.598515889
0.775	0.601576053
0.776	0.604636655
0.777	0.607697539
0.778	0.610758546
0.779	0.613819518
0.78	0.616880294
0.781	0.619940714
0.782	0.623000614
0.783	0.626059832
0.784	0.629118205
0.785	0.632175565
0.786	0.635231749
0.787	0.638286588
0.788	0.641339915
0.789	0.644391561
0.79	0.647441356
0.791	0.650489129
0.792	0.653534709
0.793	0.656577924
0.794	0.659618599
0.795	0.662656562
0.796	0.665691636
0.797	0.668723647
0.798	0.671752417
0.799	0.67477777
0.8	0.375809638
0.801	0.378837081
0.802	0.381879583
0.803	0.384937106
0.804	0.388009606
0.805	0.391097038
0.806	0.394199355
0.807	0.397316508
0.808	0.400448446
0.809	0.403595116
0.81	0.406756462
0.811	0.409932425
0.812	0.413122947
0.813	0.416327964
0.814	0.419547412
0.815	0.422781224
0.816	0.42602933
0.817	0.429291659
0.818	0.432568135
0.819	0.435858683
0.82	0.439163222
0.821	0.442481671
0.822	0.445813946
0.823	0.449159959
0.824	0.452519621
0.825	0.455892839
0.826	0.459279518
0.827	0.46267956
0.828	0.466092865
0.829	0.469519329
0.83	0.472958846
0.831	0.476411306
0.832	0.479876598
0.833	0.483354607
0.834	0.486845214
0.835	0.490348298
0.836	0.493863736
0.837	0.497391399
0.838	0.500931159
0.839	0.50448288
0.84	0.508046426
0.841	0.511621658
0.842	0.515208432
0.843	0.518806601
0.844	0.522416016
0.845	0.526036523
0.846	0.529667965
0.847	0.533310181
0.848	0.536963009
0.849	0.54062628
0.85	0.544299824
0.851	0.547983465
0.852	0.551677026
0.853	0.555380325
0.854	0.559093175
0.855	0.562815387
0.856	0.566546767
0.857	0.570287118
0.858	0.574036239
0.859	0.577793924
0.86	0.581559963
0.861	0.585334144
0.862	0.589116248
0.863	0.592906054
0.864	0.596703335
0.865	0.600507862
0.866	0.6043194
0.867	0.60813771
0.868	0.611962548
0.869	0.615793667
0.87	0.619630815
0.871	0.623473734
0.872	0.627322163
0.873	0.631175837
0.874	0.635034483
0.875	0.638897828
0.876	0.64276559
0.877	0.646637484
0.878	0.65051322
0.879	0.654392504
0.88	0.658275034
0.881	0.662160507
0.882	0.666048611
0.883	0.669939032
0.884	0.673831449
0.885	0.677725536
0.886	0.681620962
0.887	0.685517391
0.888	0.689414479
0.889	0.693311881
0.89	0.697209243
0.891	0.701106207
0.892	0.705002407
0.893	0.708897475
0.894	0.712791035
0.895	0.716682704
0.896	0.720572096
0.897	0.724458816
0.898	0.728342467
0.899	0.732222642
0.9	0.34867844
0.901	0.352572074
0.902	0.356504795
0.903	0.360476953
0.904	0.364488898
0.905	0.368540985
0.906	0.37263357
0.907	0.376767012
0.908	0.380941673
0.909	0.38515792
0.91	0.389416118
0.911	0.39371664
0.912	0.398059858
0.913	0.402446148
0.914	0.406875891
0.915	0.411349469
0.916	0.415867265
0.917	0.42042967
0.918	0.425037073
0.919	0.429689869
0.92	0.434388454
0.921	0.43913323
0.922	0.443924598
0.923	0.448762967
0.924	0.453648743
0.925	0.458582341
0.926	0.463564176
0.927	0.468594667
0.928	0.473674235
0.929	0.478803305
0.93	0.483982307
0.931	0.489211672
0.932	0.494491833
0.933	0.499823231
0.934	0.505206305
0.935	0.510641502
0.936	0.516129268
0.937	0.521670056
0.938	0.52726432
0.939	0.532912519
0.94	0.538615114
0.941	0.544372571
0.942	0.550185357
0.943	0.556053946
0.944	0.561978813
0.945	0.567960438
0.946	0.573999302
0.947	0.580095892
0.948	0.586250699
0.949	0.592464215
0.95	0.598736939
0.951	0.605069371
0.952	0.611462016
0.953	0.617915382
0.954	0.624429981
0.955	0.63100633
0.956	0.637644947
0.957	0.644346357
0.958	0.651111087
0.959	0.657939669
0.96	0.664832636
0.961	0.671790529
0.962	0.67881389
0.963	0.685903267
0.964	0.69305921
0.965	0.700282274
0.966	0.707573019
0.967	0.714932007
0.968	0.722359807
0.969	0.729856988
0.97	0.737424127
0.971	0.745061803
0.972	0.7527706
0.973	0.760551107
0.974	0.768403915
0.975	0.776329621
0.976	0.784328826
0.977	0.792402135
0.978	0.800550159
0.979	0.80877351
0.98	0.817072807
0.981	0.825448673
0.982	0.833901736
0.983	0.842432627
0.984	0.851041982
0.985	0.859730442
0.986	0.868498653
0.987	0.877347265
0.988	0.886276932
0.989	0.895288314
0.99	0.904382075
0.991	0.913558883
0.992	0.922819412
0.993	0.93216434
0.994	0.94159435
0.995	0.95111013
0.996	0.960712374
0.997	0.970401777
0.998	0.980179043
0.999	0.99004488
1	0
};

\end{axis}
\end{tikzpicture}
\hskip 10pt%
\begin{tikzpicture}
\def\nna{100}
\begin{axis}[
    x tick style={color=black},
    y tick style={color=black},
    axis lines = left,
    xlabel = $p$,
    legend style={
      cells={anchor=east},
      legend pos=outer north east,},
    xmin=0,
    xmax=1.05,
    ymin=0.00,
    ymax=1.05,
    domain=0:1
]

\addplot [dotted,
    samples=10, 
    color=black,
    domain=1/\nna:1,
    ]
    {1/4};

\addplot [densely dotted,
    samples=101, 
    color=black,
    domain=0:1,
    ]
table [x=x,y=y] {
x junk y
0	0	0
0.01	0.095617925	0.264238021
0.02	0.182927193	0.323314378
0.03	0.262575873	0.35275079
0.04	0.335167364	0.371135934
0.05	0.401263061	0.384000872
0.06	0.461384886	0.393645967
0.07	0.516017693	0.401220764
0.08	0.565611546	0.407371774
0.09	0.610583882	0.412494079
0.1	0.263901071	0.416844488
0.11	0.302790757	0.420598102
0.12	0.341724966	0.423878954
0.13	0.380369185	0.426777686
0.14	0.418440037	0.429362275
0.15	0.455700176	0.431684835
0.16	0.491953574	0.433786073
0.17	0.527041154	0.435698312
0.18	0.560836778	0.437447583
0.19	0.593243538	0.439055115
0.2	0.322200474	0.440538415
0.21	0.352558644	0.441912065
0.22	0.383119706	0.443188318
0.23	0.413717274	0.444377557
0.24	0.444194949	0.44548864
0.25	0.474407196	0.446529176
0.26	0.504219988	0.447505743
0.27	0.533511224	0.448424055
0.28	0.562170964	0.449289103
0.29	0.590101489	0.450105266
0.3	0.350389282	0.450876399
0.31	0.377243308	0.451605915
0.32	0.404362553	0.45229684
0.33	0.431632042	0.452951867
0.34	0.458938832	0.453573401
0.35	0.486172984	0.454163594
0.36	0.513228434	0.454724374
0.37	0.540003793	0.455257473
0.38	0.566403033	0.455764451
0.39	0.592336096	0.45624671
0.4	0.366896742	0.456705514
0.41	0.392172822	0.457142003
0.42	0.417774943	0.457557204
0.43	0.443609379	0.45795204
0.44	0.469581262	0.458327341
0.45	0.495595408	0.458683854
0.46	0.52155713	0.459022242
0.47	0.547373025	0.459343098
0.48	0.572951744	0.459646945
0.49	0.598204723	0.45993424
0.5	0.376953125	0.460205381
0.51	0.401795277	0.460460706
0.52	0.427048256	0.460700495
0.53	0.452626975	0.460924976
0.54	0.47844287	0.46113432
0.55	0.504404592	0.461328646
0.56	0.530418738	0.461508018
0.57	0.556390621	0.461672448
0.58	0.582225057	0.46182189
0.59	0.607827178	0.461956242
0.6	0.382280602	0.462075341
0.61	0.407663904	0.462178962
0.62	0.433596967	0.462266814
0.63	0.459996207	0.462338532
0.64	0.486771566	0.462393674
0.65	0.513827016	0.462431714
0.66	0.541061168	0.462452029
0.67	0.568367958	0.462453897
0.68	0.595637447	0.462436477
0.69	0.622756692	0.462398796
0.7	0.382782786	0.462339736
0.71	0.409898511	0.462258008
0.72	0.437829036	0.462152127
0.73	0.466488776	0.462020387
0.74	0.495780012	0.461860815
0.75	0.525592804	0.461671132
0.76	0.555805051	0.461448696
0.77	0.586282726	0.461190429
0.78	0.616880294	0.460892734
0.79	0.647441356	0.460551381
0.8	0.375809638	0.46016137
0.81	0.406756462	0.459716749
0.82	0.439163222	0.459210377
0.83	0.472958846	0.458633616
0.84	0.508046426	0.457975907
0.85	0.544299824	0.457224206
0.86	0.581559963	0.45636219
0.87	0.619630815	0.455369137
0.88	0.658275034	0.454218291
0.89	0.697209243	0.45287444
0.9	0.34867844	0.451290165
0.91	0.389416118	0.449399864
0.92	0.434388454	0.447109751
0.93	0.483982307	0.444280242
0.94	0.538615114	0.440692724
0.95	0.598736939	0.435981301
0.96	0.664832636	0.429475569
0.97	0.737424127	0.419775083
0.98	0.817072807	0.403271711
0.99	0.904382075	0.366032341
1	0	0
};

\addplot [solid,
    samples=500, 
    color=black,
    domain=1/\nna:1,
    ]
    {(1/2)-sqrt(\nna/(2*pi*floor(\nna*x)*(\nna-floor(\nna*x))))};
 
\addplot [densely dashed,
    samples=200, 
    color=black,
    domain=1/\nna:1-1/\nna,
    ]
    {sqrt(\nna*x*(1-x))/(2*sqrt(2)*(sqrt(\nna*x*(1-x)+1)+1))};
 
\end{axis}
\end{tikzpicture}

\caption{Comparison of the previous bounds of~\cite{GreenbergM14} (Lemma~\ref{lprobfeigebin1}, sparsely dotted lines) and~\cite{PelekisR16} (Lemma~\ref{lprobfeigebin2}, dashed lines), of our bound (equation~\eqref{eq:ourresult}, solid lines), and of the true value for $\Pr[\Bin(n,p) > np]$ (dotted lines) for $n=10$ (left) and $n=100$ (right). To increase the readability, for $n=100$ the true value is only depicted at the local minima $\{0,\frac 1n, \frac 2n,...,1\}$.}\label{figbounds}
\end{figure}

\section{Preliminaries}

All notation we shall use is standard and should need not much additional explanation. We denote by $\N := \{1, 2, \dots\}$ the positive integers. For intervals of integers, we write $[a..b] := \{x \in \Z \mid a \le x \le b\}$. We use the standard definition $0^0 := 1$ (and not $0^0 = 0$). 

It is well-known that $(1 - \tfrac 1r)^r$ is monotonically increasing and that $(1 - \tfrac 1r)^{r-1}$ is monotonically decreasing in $r$ (and that both converge to $\frac 1e$). We need two slightly stronger statements in this work (Lemma~\ref{lmono}~\ref{itlmono1} and~\ref{itlmono3}).
\begin{lemma}\label{lmono}
\begin{enumerate}
	\item \label{itlmono1} For all $\alpha \ge 0$, the expression $(1-\frac 1x)^{x-0.5+\alpha}$ is increasing for $x \ge 1$.
	\item \label{itlmono2} For all $\alpha \ge 0$, the expression $(1+\frac 1x)^{x+0.5+\alpha}$ is decreasing for $x > 0$.
	\item \label{itlmono3} The expression $(1-\frac1x)^x+(1-\frac1x)^{x-1} = 2(1-\frac1{2x})(1-\frac1x)^{x-1}$ is decreasing for $x \ge 1$.
\end{enumerate}
\end{lemma}

\begin{proof}
  To prove the first part, it suffices to show that $f(x) := (1-\frac 1x)^{x-0.5}$ is increasing for $x \ge 1$. It is obvious that $0=f(1)<f(x)$ for all $x > 1$, so we can concentrate on the case $x>1$. We show that $\ln(f(x))$ is increasing for $x > 1$. Using the series expansion of the natural logarithm, we compute $\ln(f(x)) = (x-\frac12)\ln(1-\frac1x) = (x-\frac12) \sum_{i=1}^\infty (-\frac{1}{ix^i}) = \sum_{i=0}^\infty (-\frac{1}{(i+1)x^i}) + \sum_{i=1}^\infty \frac{1}{2ix^i} = -1 - \sum_{i=1}^\infty \frac{i-1}{2i(i+1)}\frac{1}{x^i}$, which is a sum of constant and increasing functions.
  
  The second claim follows from noting that the reciprocal of our expression, $\frac{1}{(1+\frac 1x)^{x+0.5+\alpha}} = (1-\frac{1}{x+1})^{(x+1)-0.5+\alpha}$, is increasing for $x \ge 0$ by the first part. 
  
  The third claim follows along similar arguments as the first one. Since $f(x) = 2(1-\frac1{2x})(1-\frac1x)^{x-1}$ is continuous in $x=1$, it suffices to show the claim for $x > 1$. For $x > 1$, we regard $\ln(\frac 12 f(x))$ and compute $\ln(\frac 12 f(x)) = \ln(1-\frac1{2x}) + (x-1)\ln(1-\frac1x) = -1 + \sum_{i=1}^\infty (\frac1{i+1} - \frac{1}{2^i}) \frac 1i \frac 1{x^i}$, which is a sum of constant and decreasing functions.
\end{proof}

\begin{figure}
\quick{
\begin{tikzpicture}
\begin{axis}[
    x tick style={color=black},
    y tick style={color=black},
    axis lines = left,
    xlabel = $x$,
    ylabel = {$f(x)$},
    legend style={
      cells={anchor=east},
      legend pos=outer north east,},
    xmin=1.01,
    xmax=11.0,
    domain=1.01:11.0
]

\addplot [dotted,
    samples=200, 
    color=black,
    ]
    {(1-1/x)^(x-1)};
\addlegendentry[right]{$f(x) = (1-\frac1x)^{x-1}$}

\addplot [densely dotted,
    samples=200, 
    color=black,
    ]
    {((1-1/x)^(x-1)+(1-1/x)^(x))/2};
\addlegendentry[right]{$f(x) = \frac12 (1-\frac1x)^{x-1}+ \frac12(1-\frac1x)^{x}$}
 
\addplot [densely dashed,
    samples=300, 
    color=black,
    ]
    {(1-1/x)^(x-0.5)};
\addlegendentry[right]{$f(x) = (1-\frac1x)^{x-0.5}$}
 
\addplot [solid,
    samples=300, 
    color=black,
    ]
    {(1-1/x)^x};
\addlegendentry[right]{$f(x) = (1-\frac1x)^x$}
 
%
 
\end{axis}
\end{tikzpicture}
}
\caption{Plots related to Lemma~\ref{lmono}.}\label{figprobesbm}
\end{figure}
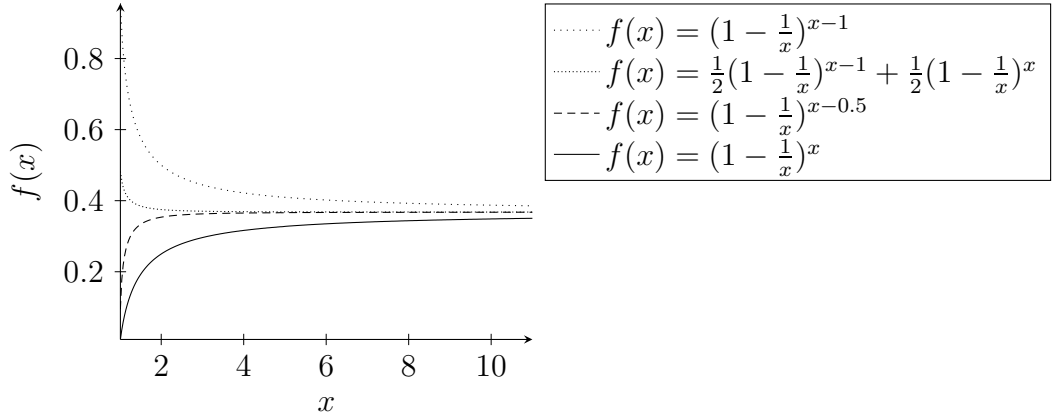

From the following version of Stirling's formula we obtain an estimate for binomial coefficients and from that an estimate for $\Pr[\Bin(n,\frac kn) = k]$.

\begin{theorem}[Robbins~\cite{Robbins55}]\label{tprobstirling}
  For all $n \in \N$, \[n! = \sqrt{2\pi n} (\tfrac ne)^{n} R_n,\] where $1 < \exp(\frac{1}{12n+1}) < R_n < \exp(\frac{1}{12n}) < 1.08690405$.
\end{theorem}

\begin{corollary}\label{corprobstirling}
  For all $n \in \N$ and $k \in [1..n-1]$, \[\binom{n}{k} = \frac{1}{\sqrt{2\pi}} \, \sqrt{\frac{n}{k(n-k)}} \, \bigg(\frac{n}{k}\bigg)^k \bigg(\frac{n}{n-k}\bigg)^{n-k} R_{nk},\] where $0.88102729... = \exp(-\frac 16 + \frac{1}{25}) \le \exp(-\tfrac{1}{12k} - \tfrac{1}{12(n-k)} + \tfrac{1}{12n+1}) < R_{nk} < \exp(-\tfrac{1}{12k+1} - \tfrac{1}{12(n-k)+1} + \tfrac{1}{12n}) < 1$.
\end{corollary}

\begin{lemma}\label{lprobmaxbinom}
 Let $n \in \N$ and $k \in [1..n-1]$. Then $\Pr[X = k] < \frac{1}{\sqrt{2\pi}} \, \sqrt{\frac{n}{k(n-k)}}$.
\end{lemma}

\begin{proof}
  With Corollary~\ref{corprobstirling} we estimate
  \begin{align*}
  \Pr[X = k] &\le \Pr[Y=k]\\
  & = \binom{n}{k} \left(\frac kn\right)^k \left(1-\frac kn\right)^{n-k}\\
  &< \frac{1}{\sqrt{2\pi}} \, \sqrt{\frac{n}{k(n-k)}}.
  \end{align*}
\end{proof}

A simple coupling argument establishes the natural fact that a binomial distribution with smaller $p$-value is dominated (in a strict sense) by one with larger $p$-value. Very similar results were used and proven also in~\cite{GreenbergM14} (Lemma~1) and~\cite{PelekisR16} (Lemma~2.4), however, with more complicated and less intuitive proofs (via differentiating the function $p \mapsto \Pr[\Bin(n,p) \ge np]$ in~\cite{GreenbergM14} and via arguing that $\Bin(n,p)$ is smaller than $\Bin(n,q)$ in the likelihood order in~\cite{PelekisR16}).

\begin{lemma}\label{ldom}
  Let $n \in \N$ and $0 \le p < q \le 1$. Let $X \sim \Bin(n,p)$ and $Y \sim \Bin(n,q)$. Then for all $k \in [0..n]$, $\Pr[X \ge k] < \Pr[Y \ge k]$.
\end{lemma}

\begin{proof}
  Let $R_1, \dots, R_n$ be independent random variables uniformly distributed in $[0,1)$. Let 
  \begin{align*}
  \tilde X &:= |\{i \in [1..n] \mid R_i < p\}|,\\
  \tilde Y &:= |\{i \in [1..n] \mid R_i < q\}|.
  \end{align*}
  By construction, we have $\tilde X \le \tilde Y$ and thus $\Pr[\tilde X \ge k] \le \Pr[\tilde Y \ge k]$. Since the event ``$\tilde X = k-1$ and $\tilde Y = k$'' appears with positive probability, we even have $\Pr[\tilde X \ge k] < \Pr[\tilde Y \ge k]$. Also by construction, $\tilde X \sim X$ and $\tilde Y \sim Y$, so the previous statement is also valid for $X$ and $Y$. 
\end{proof}

\section{Proofs of Our Results}

We are now ready to prove our results.

\begin{lemma}\label{lmain}
  Let $n \in \N$ and $\frac 1n \le p < 1$. Let $X \sim \Bin(n,p)$. Let $k = \lfloor np \rfloor$. Then 
  \begin{align}
  \Pr[X > E[X]] 
  &> \frac 12 - \sqrt{\frac{n}{2\pi k (n - k)}} =: g(n,k).\label{eq:probxex2}
  \end{align}  
\end{lemma}

\begin{proof}
  We compute
  \begin{align*}
  \Pr[X > E[X]] 
  &= \Pr[X \ge k+1] \\
  &\ge \Pr[\Bin(n,\tfrac kn) \ge k+1] \\
  &= \Pr[\Bin(n,\tfrac kn) \ge k] - \Pr[\Bin(n,\tfrac kn) = k] \\
  &> \frac 12 - \sqrt{\frac{n}{2\pi k (n - k)}}.
  \end{align*}
  Here the first inequality stems from the natural stochastic domination relation between binomial distributions with different success probabilities (Lemma~\ref{ldom}). Last estimate uses (i) the well-known fact that a binomial distribution with integral expectation has this expectation as (unique) median~\cite{Neumann66} and (ii) the estimate from Lemma~\ref{lprobmaxbinom}. 
\end{proof}

Note that when $n$ is fixed, then $g(n,k)$ is minimal for $k=1$ and $k=n-1$. Also, $g(n,1)=g(n,n-1)$ is monotonically increasing in $n$. Hence for all $n \ge 3$, $\frac 1n \le p < 1$, and $k := \lfloor np \rfloor$,  we have $\Pr[X > E[X]] > g(n,k) \ge g(n,1) \ge g(3,1) > 0.0113$. For $n = 2$ and consequently $p \ge \frac 12$, we compute $\Pr[X > E[X]] \ge \Pr[X=2]  = p^2 \ge \frac 14$. Hence Lemma~\ref{lmain} immediately gives a constant lower bound for the probability to exceed the expectation. Since we expect that the precise constant of $\frac 14$ shown below is not important in several applications, e.g., in the runtime analysis of algorithms, we formulate this elementary result explicitly. We add the trivial observation that this result, by possibly lowering the constant, can be extended to smaller values of $p$ as long as they are at least $\frac \eps n$ for some $\eps > 0$.

\begin{corollary}\label{cormain}
  Let $\eps > 0$. Then for all $n \in \N$, $\frac \eps n \le p < 1$, and $X \sim \Bin(n,p)$, we have $\Pr[X > E[X]] > \min\{e^{-\eps}, 0.0113\}$.
\end{corollary}

\begin{proof}
  For $p \ge \frac 1n$, the claim follows from the above discussion. For $0 < p < \frac 1n$, we  compute $\Pr[X > E[X]] = 1 - \Pr[X=0] = 1 - (1-p)^n > 1 - \exp(-pn)$, using the well-known estimate $1+x < e^x$ valid for all $x \neq 0$. 
\end{proof}  

We now show how to improve the lower bound to $\frac 14$. Since $g(n,k)$ approaches $\frac 12$ when $k$ and $n-k$ tend to infinity, it is clear that we only have to deal with ``small cases''. More specifically, since $g(20,3) = g(20,17) \ge 0.2501$, we need to regard for arbitrary $n$ the cases that $p < \frac3n$ and $p \ge 1 - \frac 3n$. In addition, we need to consider the finite number of cases where $n \le 19$. This will give the following result.

\begin{theorem}\label{tmain}
  Let $n \in \N$ and $\frac \alpha n \le p < 1$, where $\alpha := \ln(\frac 43) < 0.2877$. Let $X \sim \Bin(n,p)$. Then $\Pr[X > E[X]] \ge \frac 14$.
\end{theorem}

\begin{proof}
{For $p \in [\frac \alpha n,\frac 1n)$} with $\alpha := \ln(\frac43) <0.2877$, we compute $\Pr[X > E[X]] = \Pr[X \ge 1] =  1 - (1-p)^n > 1 - \exp(-pn) \ge 1 - \exp(-\alpha) \ge \frac 14$ by choice of $\alpha$.

We assume from now on that $p \ge \frac 1n$. Let $Y \sim \Bin(n,\frac kn)$ with $k = \lfloor np \rfloor$. By Lemma~\ref{ldom}, 
\[\Pr[X > E[X]] = \Pr[X \ge k+1] \ge \Pr[Y \ge k+1].\] Hence it suffices to show 
\begin{equation}
\Pr[X \ge k+1] \ge \tfrac 14 \mbox { with } X \sim \Bin(n,\tfrac kn) \label{eqtodo1}
\end{equation} 
for all $n \in \N$ and $k \in [1..n-1]$. To this aim, let us assume that $X \sim \Bin(n,\frac kn)$ in the remainder of this proof. We start by treating the ``small'' cases $k \in \{1,2,n-3,n-2,n-1\}$.
  
For $k=1$ and $n \ge 3$, we compute $\Pr[X \ge 2] = 1 - \Pr[X=0] - \Pr[X=1] = 1 - (1-\frac 1n)^n - (1-\frac1n)^{n-1} \ge \frac{7}{27} > 0.2592$, where the first inequality stems from the fact that $n \mapsto (1-\frac 1n)^n + (1-\frac1n)^{n-1}$ is decreasing (Lemma~\ref{lmono}) and $n \ge 3$.

For $k=2$ and $n \ge 5$, in a similar fashion we compute $\Pr[X \ge 3]  
= 1 - (1-\frac 2n)^n - 2(1-\frac 2n)^{n-1} - 2 (1-\frac 1n)(1-\frac 2n)^{n-2}
= 1 - (1-\frac 2n)^{n-1} (5+\frac{4}{(n-2) n}) > 1 - e^{-2}(5+\frac{4}{(n-2) n}) \ge 1 - e^{-2}(5+\frac{4}{15}) > 0.2872$ by noting that $(1-\frac 2n)^{n-1} = ((1 - \frac 1 {n/2})^{n/2 - 1/2})^2$ is increasing (Lemma~\ref{lmono}) and tending to $e^{-2}$.

For $k = n-1$ and $n \ge 2$, we estimate $\Pr[X \ge n] = (1-\frac 1n)^n \ge \frac 14$ using Lemma~\ref{lmono} and $n \ge 2$. 

For $k = n-2$ and $n \ge 4$, we estimate $\Pr[X \ge n-1] = 2 (1-\frac2n)^{n-1} + (1-\frac 2n)^n \ge 2 (\frac 12)^3 + (\frac12)^4 = 0.3125$ using Lemma~\ref{lmono} and $n \ge 4$. 
  
For $k = n-3$ and $n \ge 6$, we estimate $\Pr[X \ge n-2] = \frac92 (1-\frac1n) (1-\frac 3n)^{n-2} + 3 (1-\frac3n)^{n-1} + (1-\frac 3n)^n > \frac92 (1-\frac 3n)^{n-1} + 3 (1-\frac3n)^{n-1} + (1-\frac 3n)^n \ge 0.25$, using Lemma~\ref{lmono} and $n \ge 6$. 
  
With this case distinction, we have proven~\eqref{eqtodo1} for all $n \in \N$ and $k \in \{1,2,n-3,n-2,n-1\} \cap [1..n-1]$. Since $g(20,3) = g(20,17) > 0.25$, the concavity of $k \mapsto g(n,k)$ gives $g(20,k) > 0.25$ for all $k \in [3..17]$. Since for all $n \in \N$ and $k \in [1..n-1]$ we have $g(n+1,k)>g(n,k)$ and $g(n+1,k+1) > g(n,k)$, we have $g(n,k) > 0.25$ for all $n \ge 20$ and $k \in [3..n-3]$, which proves~\eqref{eqtodo1} for all $n \ge 20$. 

Hence it remains to show~\eqref{eqtodo1} for all $n \le 19$ and $k \in [3..n-4]$. In principle, these $91$ cases can easily be checked in an automated fashion. If we prefer a human-readable proof, we can argue as follows. 

For the case $k=3$ and $n \ge 7$, we compute $\Pr[X \ge 4] \ge 1 - (1-\frac3n)^n - 3(1-\frac3n)^{n-1} - \frac92 (1-\frac1n)(1-\frac3n)^{n-2} - \frac92 (1-\frac1n)(1-\frac2n)(1-\frac3n)^{n-3}$. By Lemma~\ref{lmono}, $(1-\frac3n)^n$ and $3(1-\frac3n)^{n-1}$ are increasing in $n$ and tend to $e^{-3}$. We shall argue that $\frac92 (1-\frac1n)(1-\frac3n)^{n-2} + \frac92 (1-\frac1n)(1-\frac2n)(1-\frac3n)^{n-3}$ is decreasing, hence it is at most $\frac{124416}{235298}$ for $n \ge 7$. Hence $\Pr[X \ge 4] \ge 1 - 4e^{-3} - \frac{124416}{235298} \ge 0.2720$. To see that $\frac92 (1-\frac1n)(1-\frac3n)^{n-2} + \frac92 (1-\frac1n)(1-\frac2n)(1-\frac3n)^{n-3}$ is decreasing, we rewrite this term as $f(n) = 9 (1-\frac1n)(1-\frac5{2n})(1-\frac3n)^{n-3}$. In a similar fashion as in the proof of Lemma~\ref{lmono}, we see that $\ln(\frac 19 f(n)) = -3 + \sum_{i=1}^\infty \frac{3^{i+1} - (\frac52)^{i} (i+1) - (i+1)}{i (i+1)} n^{-i}$, which is a sum of constant and decreasing function.

For the case $k=n-4$ and $n \ge 8$, we compute $\Pr[X \ge n-3] = \frac{4^3}{6} (1-\frac1n)(1-\frac2n)(1-\frac4n)^{n-3} + 8 (1-\frac1n)(1-\frac4n)^{n-2} + 4(1-\frac4n)^{n-1} + (1-\frac4n)^4 > \frac{4^3}{6} (1-\frac1n)(1-\frac4n)^{n-2} + 8 (1-\frac1n)(1-\frac4n)^{n-2} + 4(1-\frac4n)^{n-1} + (1-\frac4n)^4$, which is increasing in $n$ by Lemma~\ref{lmono}. Using $n \ge 8$, we conclude $\Pr[X \ge n-3] \ge 0.2903$.

We now note that $g(12,4) = g(12,8) > 0.25$ and that $g(11,5) = g(11,6) > 0.25$. With the same monotonicity arguments as above, this solves all cases $(n,k)$ which can be written as $(n,k) = (n_0+i+j,k_0+j)$ with $i,j \in \N \cup \{0\}$ and $(n_0,k_0) \in \{(12,4),(11,5),(11,6),(12,8)\}$. This leaves to check only the cases $(n,k) \in \{(9,4),(10,4),(11,4),(10,5)\}$. These are best computed by hand, e.g., $\Pr[\Bin(9,\frac49) \ge 5] \ge \Pr[\Bin(9,\frac49) = 5] + \Pr[\Bin(9,\frac49) = 6] \ge 0.2081 + 0.1110 = 0.3191$.\footnote{More precisely, for $(n,k) = (9,4),(10,4),(11,4),(10,5)$ we have $\Pr[\Bin(n,\frac kn) \ge k+1] \approx 0.3655, 0.3668, 0.3678, 0.3769$, respectively, where all values have been rounded down to the nearest multiple of $0.0001$.}
\end{proof}

Since~\cite{GreenbergM14} showed a probability of strictly more than $\frac 14$ (albeit for the weaker event $X \ge E[X]$), let us remark that our proof above also shows $\Pr[X > E[X]] > \frac 14$ for all $n \in \N$ and $\ln(\frac 43)/n \le p < 1$ except $(n,p) = (2, \frac 12)$, where indeed $\Pr[X > E[X]] = \frac 14$. The proof above does not exclude $\Pr[\Bin(n,\frac kn) = k+1] = \frac 14$ only for the case that $n \ge 2$ and $k = n-1$. The success probability shown above is (strictly) increasing in $n$, so the one remaining case is $(n,k)=(2,1)$, where indeed we do not have $\Pr[X > E[X]] > \frac 14$. By the strict domination result of Lemma~\ref{ldom}, for $p \notin \{\frac 1n, \frac 2n, \dots, \frac{n-1}n\}$ we always have $\Pr[X > E[X]] > \frac 14$.

\subsection*{Exceeding the Expectation by More Than One}

We end this section with a short proof of the fact that binomial random variables exceed their expectation also by more than one with constant probability (obviously only when $p < 1 - \frac 1n$). For this problem, not much previous work exists. Pelekis~\cite{Pelekis16} shows the following estimate for exceeding the expectation by general amounts.

\begin{theorem}\label{tpelekis}
  Let $n \in \N$, $0 < p < 1$, and $np < k \le n-1$. With $\ell := \lfloor \frac{k-np}{1-p} \rfloor$, we have 
  \[\Pr[X \ge k] \ge \frac{p^{2\ell+2}}{2} \frac{\binom{n}{\ell+1}}{\binom{k}{\ell+1}}.\]
\end{theorem}

For the case of exceeding the expectation by more than one, we show the following results. As Figure~\ref{figplusone} indicates, Pelekis' and our bounds do not compare easily, but ours give the more uniform results.

\begin{figure}
\quick{
\begin{tikzpicture}
\def\nn{20}
\begin{axis}[
    x tick style={color=black},
    y tick style={color=black},
    axis lines = left,
    xlabel = $p$,
    legend style={
      cells={anchor=east},
      legend pos=outer north east,},
    xmin=0,
    xmax=1.05,
    ymin=0.00,
    domain=0:1,
    declare function={
      k(\t)=floor(\nn*\t);
      ktarget(\t)=floor(\nn*\t)+2;
      L(\t)=floor((ktarget(\t)-\nn*\t)/(1-\t));
      }
]

\addplot [solid,
    samples=500, 
    color=black,
    domain=1/\nn:1,
    ]
    {(1/2)-2*sqrt(\nn/(2*pi*k(x)*(\nn-k(x))))};

\addplot [densely dashed,
    samples=500, 
    color=black,
    domain=1/\nn:1-1/\nn,
    ]
    {0.25 - sqrt(\nn/(2*pi*(k(x)+1)*(\nn - k(x) - 1))) * (1 - 1/(k(x)+1))^(k(x)+1) * (1 + 1/(\nn-k(x)-1))^(\nn-k(x)-1)};

\addplot [dotted,
    samples=100, 
    color=black,
    domain=1/\nn:1-1/\nn,
    ]
    {0.037};

\addplot [densely dotted,
    samples=1500, 
    color=black,
    domain=1/\nn:1-1/\nn,
    ]
    {x^(2*L(x)+2) / 2 * factorial(\nn) / factorial (\nn - L(x) -1) / factorial(ktarget(x)) * factorial (ktarget(x)-L(x)-1)};

\end{axis}
\end{tikzpicture}
\hskip 10pt%
\begin{tikzpicture}
\def\nna{100}
\begin{axis}[
    x tick style={color=black},
    y tick style={color=black},
    axis lines = left,
    xlabel = $p$,
    legend style={
      cells={anchor=east},
      legend pos=outer north east,},
    xmin=0.001,
    xmax=0.999,
    ymin=0.00,
    domain=0.001:0.999,
    declare function={
      ktarget(\t)=floor(\nna*\t)+2;
      L(\t)=floor((ktarget(\t)-\nna*\t)/(1-\t));
      }
]

\addplot [solid,
    samples=500, 
    color=black,
    domain=1/\nna:1,
    ]
    {(1/2)-2*sqrt(\nna/(2*pi*floor(\nna*x)*(\nna-floor(\nna*x))))};
 
\addplot [densely dashed,
    samples=500, 
    color=black,
    domain=1/\nna:1-1/\nna,
    ]
    {0.25 - sqrt(\nna/(2*pi*(floor(\nna*x)+1)*(\nna - floor(\nna*x) - 1))) * (1 - 1/(floor(\nna*x)+1))^(floor(\nna*x)+1) * (1 + 1/(\nna-floor(\nna*x)-1))^(\nna-floor(\nna*x)-1)};

\addplot [dotted,
    samples=100, 
    color=black,
    domain=1/\nna:1-1/\nna,
    ]
    {0.037};

\addplot [densely dotted,
    samples=1500, 
    color=black,
    domain=1/\nna:1-1/\nna,
    ]
    {x^(2*L(x)+2) / 2 * factorial(\nna) / factorial (\nna - L(x) -1) / factorial(ktarget(x)) * factorial (ktarget(x)-L(x)-1)};

\end{axis}
\end{tikzpicture}
}
\caption{Comparison of the lower bounds for $\Pr[\Bin(n,p) > E[X]+1]$ stemming from Theorem~\ref{tpelekis}~\cite{Pelekis16} (dotted lines) and our Theorem~\ref{tplusone}~(a) (solid lines),~(b) (dashed lines), and~(c) (sparsely dotted lines) for $n=20$ (left) and $n=100$ (right). }\label{figplusone}
\end{figure}
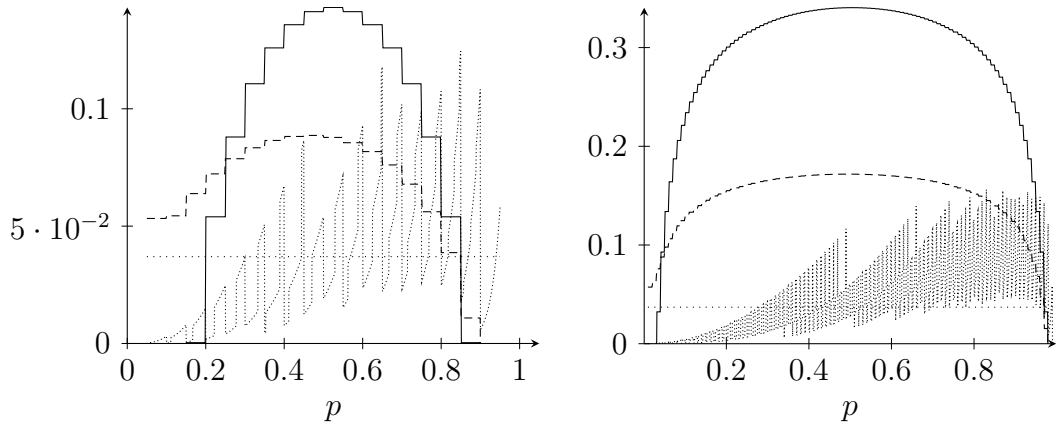

\begin{theorem}\label{tplusone}
  Let $n \in \N_{\ge 3}$ and $\frac 1n \le p < 1-\frac 1n$. Let $X \sim \Bin(n,p)$. Let $k = \lfloor np \rfloor$. Then
  \begin{enumerate}
	\item $\Pr[X > E[X]+1] \ge 0.5 - 2 \sqrt{\frac{n}{2\pi k (n - k)}}$;
	\item $\Pr[X > E[X]+1] \ge 0.25 -  \sqrt{\frac{n}{2\pi (k+1) (n - k - 1)}} \left(1 - \frac 1 {k+1}\right)^{k+1} \left(1 + \frac 1 {n-k-1}\right)^{n-k-1}$;
	\item $\Pr[X > E[X]+1] \ge 0.0370$.
  \end{enumerate}
\end{theorem}

\begin{proof}
  Let $Y \sim \Bin(n,\tfrac kn)$. By Lemma~\ref{ldom}, we have
  \begin{align*}
  \Pr[X > E[X]+1] = \Pr[X \ge k+2] \ge \Pr[Y \ge k+2],
  \end{align*}
  so we can assume in the following that $p = \frac kn$ with $k \in [1..n-2]$ and show our claims for $\Pr[X \ge k+2]$.
  
  Similar to the proof of Lemma~\ref{lmain}, we compute
  \begin{align*}
  \Pr[X \ge k+2]  &= \Pr[X \ge k] - \Pr[X = k]  - \Pr[X = k+1] \\
  &\ge \Pr[X \ge k] - 2 \Pr[X = k] \\
  &> \frac 12 - 2 \sqrt{\frac{n}{2\pi k (n - k)}},
  \end{align*}
  where the estimate $\Pr[X = k] \ge \Pr[X = k+1]$ used in the penultimate inequality either follows from a simple computation or from the well-known fact that the mode of $\Bin(n,\frac kn)$ is $k$.
  
  By using Theorem~\ref{tmain} and Corollary~\ref{corprobstirling}, we compute   
  \begin{align*}
  \Pr[&X \ge k+2] = \Pr[X > E[X]] - \Pr[X = E[X]+1] \nonumber\\
  &\ge \frac 14 - \sqrt{\frac{n}{2\pi (k+1) (n - k - 1)}} \left(1-\frac 1 {k+1}\right)^{k+1} \left(1+\frac 1 {n-k-1}\right)^{n-k-1}\nonumber\\
  &:= h(n,k).\label{eq}
  \end{align*}
  
  To prove the last claim, recall that we have to show 
  \begin{equation}
  \Pr[\Bin(n,\tfrac kn) \ge k+2] \ge 0.0370 \label{eqtodo}
  \end{equation}
  for all $n \ge 3$ and $k \in [1..n-2]$. To this aim, we first exploit in two different ways the lower bound $h(n,k)$. 
 
  By rewriting 
  \[\sqrt{\frac{n}{n - k -1}} \left(1+\frac 1 {n-k-1}\right)^{n-k-1} = \sqrt{\frac{n}{n - k}} \left(1+\frac 1 {n-k-1}\right)^{n-k-1+0.5}\]
  as product of two terms which are both decreasing in $n$ for $n > k+1$, see Lemma~\ref{lmono}~\ref{itlmono2}, we observe that $h(n,k)$ is increasing for $k$ fixed and $n \ge k+2$ growing. Since $h(n,k) > 0.0370$ for $(n,k) \in \{(6,1), (9,2), (9,3), (10,4), (10,5)\},$\footnote{More precisely, we have $h(6,1) \approx 0.0391$, $h(9,2) \approx 0.0392$, $h(9,3) \approx 0.0392$, $h(10,4) \approx 0.0442$, and $h(10,5) \approx 0.0394$, where the value given is the true value rounded down to the nearest multiple of $0.0001$.} we have shown~\eqref{eqtodo} for all \[(n,k) \in \{(6+i,1), (9+i,2), (9+i,3), (10+i,4), (10+i,5) \mid i \ge 0\}.\]
  
  We now argue that $h(n,n-5)$ is increasing in $n$ for $n \ge 10$. We rewrite the relevant part $\sqrt{\frac n {n-4}} (1 - \frac 1 {n-4})^{n-4} = (1 + \frac 4 {n-4})^{1/2} (1 - \frac 1 {n-4}) (1 - \frac 1 {n-4})^{(n-4)-1}$. The last factor is known to be decreasing for $n > 4$. For the first two factors, we compute $((1 + \frac 4 {n-4})^{1/2} (1 - \frac 1 {n-4}))^2 = 1 + \frac 2 {n-4} - \frac 7 {(n-4)^2} + \frac 4 {(n-4)^3}$, which is decreasing from $n-4 \ge 6$ on. Since $h(10,5) > 0.0370$, we  have $h(n,k) > 0.0370$ for all $(n,k) \in \{(10+i,5+i) \mid i \ge 0\}$, and via the first monotonicity statement for all \[(n,k) \in \{(10+i+j,5+i) \mid i,j \ge 0\}.\]
  
  For the remaining cases, we estimate directly the probability $q(n,k) := \Pr[\Bin(n,\frac kn) \ge k+2]$ in~\eqref{eqtodo}. 
  For $k \in \{n-2, n-3, n-4\}$, we simply compute
  \begin{align*}
  q(n,n-2) &= (1-\tfrac2n)^n \ge 0.0370,\\
  q(n,n-3) &= (1-\tfrac3n)^n + 3 (1-\tfrac3n)^{n-1} \ge 0.0507,\\
  q(n,n-4) &= (1-\tfrac4n)^n + 4 (1-\tfrac4n)^{n-1} + 8 (1-\tfrac1n)(1-\tfrac4n)^{n-2} \ge 0.0579,
  \end{align*}
  where the last estimates stem from noting that all the terms involved are increasing in $n$ by Lemma~\ref{lmono}~\ref{itlmono1} and evaluating the expressions for $n=3$, $n=4$, and $n=5$, respectively. This shows~\eqref{eqtodo} for all \[(n,k) \in \{(3+i,1+i), (4+i,1+i), (5+i,1+i) \mid i \ge 0\}.\]
  
  For the last four cases $(n,k) \in \{(7,2), (8,2), (8,3), (9,4)\}$, equation~\eqref{eqtodo} is easily checked by hand: We have $q(7,2) \approx 0.1082$, $q(8,2) \approx 0.1138$, $q(8,3) \approx 0.1374$, and $q(9,4) \approx 0.1573$.
  \end{proof}

We note that, as in Corollary~\ref{cormain}, smaller values of $p$ can be admitted at the price of a smaller probability for the event ``$X > E[X]+1$''. For $p = \frac \alpha n < \frac 1n$, we have $\Pr[X > E[X]+1] = 1 - (1-p)^n - np(1-p)^{n-1} = 1 - (1-\frac \alpha n)^n - \alpha(1-\frac \alpha n)^{n-1} \ge 1 - \exp(-\alpha) - \alpha \exp(-\alpha \frac{n-1}n)$.

\subsection*{Acknowledgment}

The author would like to thank Philippe Rigollet (Massachusetts Institute of Technology (MIT)) and Carsten Witt (Danish Technical University (DTU)) for useful discussions and pointers to the literature.

\newcommand{\etalchar}[1]{$^{#1}$}


\end{document}